\documentclass[10pt]{amsart}
\usepackage[all]{xy}
\pdfoutput=1
\usepackage{amsmath, amsthm, amssymb,slashed,stmaryrd}
\usepackage{enumerate}
\usepackage{ifpdf}
\usepackage[pdftex]{graphicx}
\usepackage{tikz}
\usetikzlibrary{matrix,arrows,calc}
\usepackage{mathtools}
\usepackage{tabularx}
\usepackage{csquotes}

\usepackage[pdftex,plainpages=false,hypertexnames=false,pdfpagelabels]{hyperref}
\usepackage[capitalise]{cleveref}

\usepackage{amssymb}
\usepackage{tikz-cd}
\usetikzlibrary{arrows,decorations.markings}
\usetikzlibrary{shapes.geometric}

\usepackage{bm}
\usepackage{accents}
\usepackage[super]{nth}
\usepackage{mathrsfs}

\usepackage{appendix}

\numberwithin{equation}{subsection}
\newtheorem{thm}{Theorem}[subsection]

\newtheorem{lem}[thm]{Lemma}
\newtheorem{prop}[thm]{Proposition}

\theoremstyle{definition}

\newtheorem*{setup*}{Setup}

\newtheorem{example}[thm]{Example}
\newtheorem{defn}[thm]{Definition}
\newtheorem{rmk}[thm]{Remark}
\newtheorem{question}[thm]{Question}

\newtheorem*{conjecture*}{Conjecture}
\newtheorem*{theorem*}{Theorem}
\newtheorem*{claim*}{Claim}
\newtheorem*{corollary*}{Corollary}
\newtheorem*{notation*}{Notation}

\DeclareSymbolFont{bbold}{U}{bbold}{m}{n}
\DeclareSymbolFontAlphabet{\mathbbold}{bbold}

\def\Spec{{\rm Spec}}

\def\K3{{\rm K3}}

\def\Fil{{\rm Fil}}

\def\Fpbar{{\overline{\mb{F}}_p}}

\def\det{{\rm det}}

\def\cris{{\rm cris}}

\mathchardef\mhyphen="2D

\newcommand{\Qp}{\mathbb{Q}_p}
\newcommand{\Qpbar}{\overline{\mathbb Q}_p}

\newcommand{\Qlbar}{\overline{\mathbb{Q}}_{\ell}}

\newcommand{\Z}{\mathbb Z}
\newcommand{\Q}{\mathbb Q}

\newcommand{\MF}{\mathbf{MF}}

\makeatletter
\newcommand{\colim@}[2]{%
  \vtop{\m@th\ialign{##\cr
    \hfil$#1\operator@font colim$\hfil\cr
    \noalign{\nointerlineskip\kern1.5\ex@}#2\cr
    \noalign{\nointerlineskip\kern-\ex@}\cr}}%
}
\newcommand{\colim}{%
  \mathop{\mathpalette\colim@{}}\nmlimits@
}
\makeatother

\usepackage{color}

\newcommand\nc{\newcommand}
\nc\mf\mathfrak
\nc\mc\mathcal
\nc\mb\mathbb
\nc\mbf\mathbf
\nc\msf\mathsf
\nc\mscr\mathscr

\usepackage[style=alphabetic,backend=bibtex,maxnames=6]{biblatex}
\begin{document}
	\title{Local systems which do not come from abelian varieties}
	\author{Paul Brommer-Wierig}
	\email{p.brommer-wierig@hu-berlin.de} 
    \address{Humboldt Universität Berlin,
		Institut für Mathematik- Alg.Geo.,
		Rudower Chaussee 25
		Berlin, Germany}
	
	\author{Yeuk Hay Joshua Lam}
	\email{joshua.lam@hu-berlin.de}
	\address{Humboldt Universität Berlin,
		Institut für Mathematik- Alg.Geo.,
		Rudower Chaussee 25
		Berlin, Germany}
	\date{\today}
	
	\begin{abstract}
		For each smooth curve over a finite field, after puncturing it at finitely many points, we construct local systems on it of geometric origin which do not come from a family of abelian varieties. We do so by proving a criterion which must be satisfied by local systems which do come from abelian varieties, inspired by an analogous Hodge theoretic criterion in characteristic zero.
	\end{abstract}
	
		\maketitle
	
	\section{Introduction}
		
		\subsection{Statement of the result}
			Let $X$ be a geometrically connected scheme of finite type over a finite field $\mathbb{F}_q$ and let $\ell$ be a prime number different from $p=\mathrm{char}(\mathbb{F}_q)$. A $\overline{\Q}_\ell$-local system on $X$ is a lisse $\overline{\Q}_\ell$-sheaf on $X$ or equivalently a continuous, finite-dimensional $\Qlbar$-linear representation of $\pi^\textnormal{\'{e}t}_1(X,\bar x)$ for some choice of geometric point $\bar x$ of $X$. In many ways, a $\Qlbar$-local system is the positive characteristic analogue of a variation of Hodge structure (VHS). One way of constructing $\overline{\Q}_\ell$-local systems is the following:

            \begin{example}\label{ex: geometric-construction-of-local-systems}
                Let $f\colon Y\to X$ be a smooth proper morphism. Then for each  $i\geq 0$,  the higher étale direct image $\mathrm{R}^if_{\textnormal{\'{e}t},\ast}\overline{\Q}_\ell$ gives a  $\overline{\Q}_\ell$-local system on $X$. 
            \end{example}

            From the local systems in \cref{ex: geometric-construction-of-local-systems}, one can of course generate  many more local systems, by taking duals, direct sums, tensor products, sub- and quotient-objects--we refer to such local systems as being  \enquote{of geometric origin}.  Given $X/\mb{F}_q$, one basic question is whether one can generate every $\Qlbar$-local system of geometric origin from a relatively small collection of objects. As a motivating example, the combination of the Tate conjecture and Honda--Tate theory imply that, when $X=\Spec (\mb{F}_q)$, every $\Qlbar$-sheaf which comes from geometry is generated by that attached to an abelian variety. Our  main result says that this is a peculiarity of the situation over a point, and is not true over a positive dimensional base:
			
			\begin{thm}\label{thm: main-result}
				Let $C$ be a smooth projective curve over a finite field of characteristic $p$. Let $\ell$ be a prime number different from $p$. There exists a finite set of points $S\subseteq C$ and a $\overline{\Q}_\ell$-local system on $C\setminus S$, of geometric origin, which does not come  from a family of abelian varieties over $C\setminus S$.
			\end{thm}
   
            For the precise definition of \enquote{coming from a family of abelian varieties}, see \cref{defn: come-from-av}.

        \subsection{Discussions: \texorpdfstring{$\mb{F}_q$}{Fq} versus \texorpdfstring{$\mb{C}$}{C}}
            The situation over $\mb{F}_q$ is somewhat trickier  than that over $\mb{C}$, as we now explain. As alluded to above, it is expected that, over the point $X=\Spec(\mb{F}_q)$, the cohomology of every variety is generated from that of an abelian variety. On the other hand,   there certainly exists varieties over $\mb{C}$ whose Hodge structures are provably not generated from the Hodge structures of abelian varieties; this seems to have been first observed by Deligne \cite[\S 7.6]{Deligne-La-conjecture-de-Weil-pour-les-surfaces-K3}. 
            Schematically, we have the inclusions (recall that VHS stands for variations of Hodge structure):
            \[
                \bigg\{ 
                \parbox{8em}{VHS coming from\\abelian varieties}  
                \bigg\} \subset 
                \bigg\{\parbox{8em}{\center VHS on \\Shimura varieties}  
                \bigg\} \subset 
                \bigg\{
                \parbox{5.5em}{general VHS}  
                \bigg\}.
            \]
            That is, VHS coming from families of abelian varieties are a (strict) subset of those on Shimura varieties, which are in turn a tiny fraction of general VHS; the VHS coming from abelian varieties are therefore, in some sense, the simplest kind. Similarly, over $\mb{F}_q$, the $\Qlbar$-local systems coming from abelian varieties are arguably the simplest local systems, and \cref{thm: main-result} says that one cannot build all $\Qlbar$-local systems from these simple building blocks. Presumably, our criterion also shows that most local systems  do not come from Shimura varieties, although this question is not very well formulated since there is not yet a  theory of special fibers of Shimura varieties of exceptional type parallel to that of Shimura varieties of abelian type\footnote{although see \cite{Bakker-Shankar-Tsimerman-Integral-Canonical-Models-of-Exceptional-Shimura-Varieties} for some spectacular work in this direction}. 

            Over a general variety $X/\mb{F}_q$, we find it an extremely interesting question to try to classify local systems which come from a family of abelian varieties, or perhaps from a more general Shimura variety. A more modest goal is to find more criteria to decide whether a given local system comes from abelian varieties. 

            For example, for $X/\mb{F}_q$, it is not known whether any rank two $\Qlbar$-local system comes from a family of abelian varieties, though this is conjectured to be the case by Krishnamoorthy \cite[Conjecture 1.2]{Krishnamoorthy-Rank-2-Local-Systems-and-Barsotti-Tate-Groups}\footnote{in slightly less generality}. Note that, over $\mb{C}$, any rank two local system of geometric origin (or $\mc{O}_K$-VHS for $K$ some  number field) does indeed come from a family of abelian varieties.

            Finally, we expect that, for each fixed curve $X/\mb{F}_q$ with $\chi(X)<0$,  most local systems on $X$ of geometric origin do not come from abelian varieties. For example, in \cref{thm: main-result}, we only prove a statement after removing a finite set of points $S$.  This is presumably unnecessary, but proving it seems to require some new inputs, and we state it as a question:

            \begin{question}
                Let $X/\mb{F}_q$ be a smooth curve of negative Euler characteristic. Does there exist a local system of geometric origin on $X$, which does not come from a family of abelian varieties?
            \end{question}
        
        \subsection{Sketch of proof}
        We give some ideas of the proof of \cref{thm: main-result}, by giving an example of a local system of geometric origin which provably does not come from an abelian scheme. The case of arbitrary curves is deduced by pulling back this local system along an appropriate map. 
        
            We first recall the situation over $\mb{C}$. The analogous question is: given a $\mb{Q}$-Hodge structure $\mb{V}$, how can we tell if it comes from an abelian variety? More precisely, when is $\mb{V}$ in the Tannakian category $\langle \mathrm{H}^1(A,\Q)\rangle$ for an abelian variety $A/\mb{C}$?

            One criterion is as follows: we consider the Mumford--Tate group $\mbf{G}_{\mb{V}}$, and its Lie algebra $\mf{g}_{\mb{V}}$. The latter carries a Hodge structure, which we denote by $\underline{\mf{g}}_{\mb{V}}$ to underline this additional structure. If $\mb{V}$ comes from an abelian variety, then the Hodge decomposition must take the form 
            \begin{equation}\label{eqn:mt-decomp}
                \underline{\mf{g}}_{\mb{V}}\otimes \mb{C}\simeq \mf{g}^{-1,1}\oplus \mf{g}^{0,0}\oplus \mf{g}^{1,-1};  
            \end{equation}
            that is, the Hodge degrees are concentrated in $(-1,1), (0,0), (1,-1)$. This follows from the fact that, if $\mb{V}$ is generated from the Hodge structure of an abelian variety $A$, then there is a surjection of $\mb{Q}$-Hodge structures 
            \[
                \underline{\mf{g}}_A\rightarrow \underline{\mf{g}}_{\mb{V}},
            \]
            where $\underline{\mf{g}}_A$ denotes the Mumford--Tate Lie algebra of $A$, equipped with its canonical Hodge structure. Even better,  this necessary condition is almost sufficient. Indeed, if \eqref{eqn:mt-decomp} holds, then $\mb{V}$ corresponds to a point on a Shimura variety.

            For a local system $\mc{E}$ on a curve  $X/\mb{F}_q$, we can similarly consider the Lie algebra $\mf{g}_{\mc{E}}$ of the monodromy group, and again we may view it as a local system, which we denote by $\underline{\mf{g}}_{\mc{E}}$. However, there is no analogue of the Hodge decomposition \eqref{eqn:mt-decomp}. 

            Instead, we consider $p$-adic invariants, known as the \emph{slopes} of $\underline{\mf{g}}_{\mc{E}}$, at a closed point $x\in X$. These are rational numbers given by the $p$-adic valuations of eigenvalues of the Frobenius (at $x$) action on $\underline{\mf{g}}_{\mc{E}}$.

            
            The upshot is that, if $\mc{E}$ comes from an abelian scheme, then the slopes of $\underline{\mf{g}}_{\mc{E}}$ are heavily constrained: they must lie in $[-1, 1]$, in  analogy to the constraint provided by  \eqref{eqn:mt-decomp}. This provides a necessary criterion, as recorded  in \cref{prop: criterion-crystalline}.

            It remains to construct a local system of geometric origin over a curve which violates our criterion. We do this by analyzing the mirror quintic family of Calabi--Yau threefolds over $\mb{P}^1_{\mb{F}_q}\setminus \{0,1,\infty\}$ and using some $p$-adic Hodge theory.
          
		\subsection{Notation}
	        Throughout, we fix a prime number $p$, a finite field $\mathbb{F}_q$ of characteristic $p$ and an algebraic closure $\Fpbar$ of $\mathbb{F}_q$. A variety over a field $K$ is by definition a geometrically connected, separated scheme of finite type over $K$ and a curve is a one-dimensional variety. We will always write $X$ for a variety over $\mathbb{F}_q$ and $\ell$ denotes a prime number different from $p$.
         
		\subsection{Acknowledgments}
			The second-named author thanks Sasha Petrov for an interesting discussion on this subject in CIRM, Luminy, in the summer of 2022. We also thank Greg Baldi, Nazim Khelifa, Bruno Klingler, Raju Krishnamoorthy and Daniel Litt for their insightful comments. We are especially grateful to the referees for their extremely careful reading and helpful suggestions.
            
            During the course of this work, Brommer-Wierig was supported by the ERC grant \emph{TameHodge} (grant number: 101020009) and by the DFG Cluster of Excellence \emph{MATH+} (EXC-2046/1, project number: 390685689). Lam was supported by a Dirichlet Fellowship and the DFG Walter Benjamin grant \emph{Dynamics on character varieties and Hodge theory} (project number: 541272769). 
		
	\section{Reminder on \texorpdfstring{$\ell$}{l}-adic local systems}\label{sec: coefficient}
    
		\subsection{Tannakian categories}
			For a field $K$, we denote by $\mathbf{Vec}_K$ the Tannakian category of finite-dimensional $K$-vector spaces. Let $\mathbf{C}$ be a neutral Tannakian category over $K$ and let $\omega\colon\mathbf{C}\to\mathbf{Vec}_K$ be a $K$-linear fiber functor (\cite[Def. 2.19]{Deligne-Milne-Tannakian-Categories}). The functor $\mathrm{Aut}^\otimes(\omega)$ is representable by an affine group scheme $\mathbf{G}$ over $K$ and there exists an equivalence of tensor categories of $\mathbf{C}$ with $\mathbf{Rep}_K(\mathbf{G})$, the category of finite-dimensional linear representations of $\mathbf{G}$ over $K$, see \emph{loc. cit.} Thm. 2.11. We call $\mathbf{G}$ the \emph{Tannaka group} of $\mathbf{C}$.
    
			A \emph{Tannakian subcategory} of $\mathbf{C}$ is by definition a full abelian subcategory that is stable under $\otimes$, duals and subobjects. If $\mathbf{C}'$ is a Tannakian subcategory of $\mathbf{C}$ with Tannaka group $\mathbf{G}'$, there exists a faithfully flat morphism of $K$-group schemes $\mathbf{G}\to\mathbf{G}'$, see Cor. 2.9, Prop. 2.21.a \emph{loc. cit.} In particular $\mathbf{G}\to\mathbf{G}'$ is an epimorphism. For an object $V$ of $\mathbf{C}$ we denote by $\langle V\rangle$ the smallest Tannakian subcategory of $\mathbf{C}$ containing $V$. It is the full subcategory of $\mathbf{C}$ generated by $V$ under $\oplus, \otimes$, duals and subobjects.

            We call the Tannaka group of $\langle V\rangle$ the \emph{monodromy group} of $V$ and we denote it by $\mathbf{G}_V$. It is a linear algebraic group over $K$, see \cite[Prop. 2.20.b]{Deligne-Milne-Tannakian-Categories}. Let $\mathfrak{g}_V$ be the Lie algebra of $\mathbf{G}_V$. Thanks to the adjoint representation $\mathrm{Ad}\colon\mathbf{G}_V\to\mathbf{GL}(\mathfrak{g}_V)$ we can consider $\mathfrak{g}_V$ as an object of $\langle V\rangle$, in particular of $\mathbf{C}$.  For emphasis, we write $\underline{\mathfrak{g}}_V$ for the object of $\mathbf{C}$ given in this way. We call $\underline{\mathfrak{g}}_V$ the \emph{monodromy Lie algebra} of $V$.

        \subsection{Monodromy of \texorpdfstring{$\ell$}{l}-adic local systems}
            We define $\mathbf{LS}(X,\overline{\Q}_\ell)$ as the category of lisse $\overline{\Q}_\ell$-sheaves on $X$. It is a $\Qlbar$-linear Tannakian category. For a geometric point $\bar x$ of $X$, it admits a fiber functor
            \[
                \omega_{\bar x,\overline{\Q}_\ell}\colon\mathbf{LS}(X,\overline{\Q}_\ell)\to\mathbf{Vec}_{\overline{\Q}_\ell},
            \]
            given by sending a lisse $\Qlbar$-sheaf to its stalk at $\bar x$.
			
			\begin{defn}\label{defn: come-from-av}
				We say a $\Qlbar$-local system $\mathcal{E}$ on $X$ \emph{comes from a family of abelian varieties} if there exists an abelian scheme $f\colon A\to X$ such that $\mathcal{E}$ is an object of $\langle\mathrm{R}^1f_{\textnormal{ét},\ast}\overline{\Q}_\ell\rangle$.
			\end{defn}
			
            Let $\mathcal{E}$ be a $\Qlbar$-local system on $X$. The \emph{monodromy group} $\mathbf{G}_\mathcal{E}$ of $\mathcal{E}$ is the Tannaka group of the Tannakian category $\langle\mathcal{E}\rangle$. On the other hand, we can also define the \emph{geometric monodromy group} $\mathbf{G}_\mathcal{E}^\mathrm{geo}$ of $\mathcal{E}$ as follows: The natural projection $\pi\colon X_{\Fpbar}=X\times_{\mathbb{F}_q}\Spec(\Fpbar)\to X$ defines via pullback a tensor functor
            \[
                \pi^\ast\colon\mathbf{LS}(X,\overline{\Q}_\ell)\to\mathbf{LS}(X_\Fpbar,\overline{\Q}_\ell).
            \]
            In terms of representations, $\pi^\ast$ corresponds to the functor restricting a continuous, finite-dimensional $\Qlbar$-linear representation of $\pi_1^\textnormal{ét}(X,\bar x)$ to the normal subgroup $\pi_1^\textnormal{ét}(X_\Fpbar,\bar x)$.

            Finally, $\mathbf{G}_\mathcal{E}^\mathrm{geo}$ is the Tannaka group of the Tannakian subcategory generated by $\pi^\ast\mc{E}$ inside of $\mathbf{LS}(X_\Fpbar,\overline{\Q}_\ell)$. Let $\mf{g}^{\mathrm{geo}}_{\mc{E}}$ denote the Lie algebra of $\mathbf{G}_\mathcal{E}^\mathrm{geo}$. Note that $\mbf{G}^{\mathrm{geo}}_{\mc{E}}$ is a normal subgroup of $\mbf{G}_{\mc{E}}$, and hence $\mf{g}^{\mathrm{geo}}_{\mc{E}}$ has an action of $\mbf{G}_{\mc{E}}$. It may therefore be viewed as a $\Qlbar$-local system on $X$. We denote this $\Qlbar$-local system by $\underline{\mf{g}}^{\mathrm{geo}}_{\mc{E}}$.
   
	\section{Proof of \texorpdfstring{\cref{thm: main-result}}{the result}}\label{sec: proof}
 
		\subsection{The criterion}\label{sec: criterion}	
			We now state and prove the criterion alluded to in the introduction. We first set some notation for this subsection: We continue to write $X$ for a variety over a finite field $\mathbb{F}_q$ of characteristic $p$. We will write $x$ for a $\mathbb{F}_{q^r}$-point of $X$ and $\bar x$ for a geometric point lying over $x$. By $\mc{E}$ we denote a $\Qlbar$-local system on $X$. Finally, we fix a valuation $v$ on $\overline{\Q}$ such that $v(p)=1$.
            
            The stalk $\mathcal{E}_{\bar x}$ admits an action by the geometric $q^r$-Frobenius, denoted by $F_{\bar x}$. In all our applications,  $\mathcal{E}$ will be  \emph{algebraic}, i.e. for every closed point $x$, the eigenvalues of $F_{\bar x}$ are in $\overline{\Q}$. We define the multiset of \emph{$v$-slopes} at $x$ as the multiset consisting of the rational numbers
            \[
                \frac{v(\lambda)}{[\kappa(x)\colon\mathbb{F}_p]},
            \]
             where $\lambda$ is an eigenvalue of $F_{\bar x}$ and $\kappa(x)$ denotes the residue field at $x$.

            \begin{lem}\label{cor: estimate-for-Newton-slope-of-abelian-scheme}
                Let $f\colon A\to X$ be an abelian scheme over $\mathbb{F}_q$. Then every $v$-slope $\alpha$ of $\mathrm{R}^1f_{\textnormal{\'{e}t},\ast}\overline{\Q}_\ell$ satisfies
                \[
                    0\leq\alpha\leq 1.
                \]
            \end{lem}

            \begin{proof}
                Because the abelian variety $A_{\bar x}$ is polarizable, we deduce that for every $v$-slope $\alpha$ of $(\mathrm{R}^1f_{\textnormal{\'{e}t},\ast}\overline{\Q}_\ell)_{\bar x}=\mathrm{H}^1_\textnormal{\'{e}t}(A_{\bar x},\overline{\Q}_\ell)$,  $1-\alpha$ is also a $v$-slope. Since every eigenvalue of $F_{\bar x}$ is an algebraic integer, we see that both $\alpha\geq 0$ and $1-\alpha\geq 0$, as required.
            \end{proof}
			
			\begin{prop}\label{prop: criterion-crystalline}
	        If $\mathcal{E}$ comes from a family of abelian varieties, then the $v$-slopes of $\underline{\mathfrak{g}}_{\mathcal{E}}$ at $x$ lie in $[-1,1]$.
	        \end{prop}
	
	       \begin{proof}
		    Assume $f\colon A\to X$ is an abelian scheme over $\mathbb{F}_q$ such that $\mathcal{E}$ is an object of the Tannakian subcategory generated by $\mathcal{E}'=\mathrm{R}^1f_{\textnormal{\'{e}t},\ast}\overline{\Q}_\ell$. The inclusion of Tannakian categories $\langle\mathcal{E}\rangle\hookrightarrow\langle\mathcal{E}'\rangle$ induces an epimorphism of monodromy groups $\mathbf{G}_\mathcal{E'}\twoheadrightarrow \mathbf{G}_\mathcal{E}$. This in turn induces an epimorphism of $\overline{\Q}_\ell$-local systems $\underline{\mathfrak{g}}_\mathcal{E'}\twoheadrightarrow \underline{\mathfrak{g}}_\mathcal{E}$. Consequently, every $v$-slope of $\underline{\mathfrak{g}}_{\mathcal{E}}$ at $x$ is a $v$-slope of $\underline{\mathfrak{g}}_{\mathcal{E'}}$ at $x$. We are thus reduced to the case of $\mathcal{E}=\mathrm{R}^1f_{\textnormal{\'{e}t},\ast}\overline{\Q}_\ell$. Let us write $V=\mathrm{H}^1_\textnormal{\'{e}t}(A_{\bar x},\overline{\Q}_\ell)=(\mathrm{R}^1f_{\textnormal{\'{e}t},\ast}\Qlbar)_{\bar x}$. The natural inclusion $\underline{\mathfrak{g}}_{\mathcal{E},\bar x}\hookrightarrow\mathrm{End}(V)=V\otimes V^\vee$ is compatible with the Frobenius action. In particular, each $v$-slope of $\underline{\mathfrak{g}}_{\mathcal{E}}$ at $x$ is of the form $\alpha-\beta$, where $\alpha$ and $\beta$ are $v$-slopes of $\mathcal{E}$ at $x$. Applying \cref{cor: estimate-for-Newton-slope-of-abelian-scheme} yields the result.
	       \end{proof}

        \begin{rmk}
            We note that \cref{prop: criterion-crystalline} has a crystalline counterpart where $\Qlbar$-local systems are replaced by overconvergent $F$-isocrystals, as we wrote in the initial version of this paper. We thank one of  the referees for pointing out that the use of $F$-isocrystals is unnecessary.
        \end{rmk}


        \subsection{Filtered \texorpdfstring{$\varphi$}{phi}-modules}
            We briefly recall the notion of a filtered $\varphi$-module. For a perfect field $k$ of characteristic $p$, write $W(k)$ for the ring of Witt vectors over $k$ and let $K=W(k)[1/p]$. Denote by $\sigma: K\rightarrow K$ the automorphism induced by the absolute Frobenius on $k$.

        \begin{defn}
            A \emph{filtered $\varphi$-module} is a triple $(D, \varphi, \Fil^{\bullet})$, where 
            \begin{enumerate}[i.]
                \item $D$ is a finite dimensional $K$-vector space, $\varphi: D\rightarrow D$ is  a $\sigma$-linear automorphism;
                \item $\Fil^{\bullet}$ is a decreasing filtration by sub $K$-vector spaces  which is separated (i.e. $\cap_{i\in \mb{Z}} \Fil^i =\{0\}$) and exhaustive (i.e. $\cup_{i\in \mb{Z}} \Fil^i=D$).
            \end{enumerate}
            A morphism of filtered $\varphi$-modules is a $K$-linear map compatible with $\varphi$ and filtrations. We denote the category of filtered $\varphi$-modules by $\MF_K$.
        \end{defn}
        
        \begin{rmk}
            Suppose  $Y/K$ is a smooth proper variety with good reduction, whose special fiber we denote by $Y_k$. It is straightforward to check that, for each $i$, the triple $(\mathrm{H}^i_{\cris}(Y_k/W(k))[1/p], \varphi, \Fil^{\bullet})$ is an object of $\MF_K$; here $\varphi$ is the crystalline Frobenius, and $\Fil^{\bullet}$ is the  filtration induced by the Hodge filtration on de Rham cohomology and the comparison isomorphism  $\mathrm{H}^i_{\cris}(Y_k/W(k))[1/p]\simeq \mathrm{H}^i_{\mathrm{dR}}(Y/K)$.
        \end{rmk}
 
    \subsection{The mirror quintic family violates \texorpdfstring{\cref{prop: criterion-crystalline}}{the criterion}}
        We first give a brief reminder on the Dwork family and the construction of the  mirror quintic family of Calabi-Yau threefolds in $\mathbb{P}^4$ following \cite{Katz-Another-Look-at-the-Dwork-Family}. Let $R$ be a $\Z[1/5]$-algebra and denote by $\mu_{5,R}=\mathrm{Spec}(R[t]/(t^5-1))$ the group-scheme of fifth roots of unity over $R$. 
        
        We consider the map of smooth schemes $\pi_R'\colon Y'\to\mathbb{P}_{R}^1\setminus(\{\infty\}\cup\mu_{5,R})$  over $R$, where the fiber above $t\in \mathbb{P}_{R}^1\setminus(\{\infty\}\cup\mu_{5,R})$ is given by the hypersurface 
        \[
        X_0^5+\cdots+X_4^5-5tX_0\cdots X_4
        \]
        in $\mb{P}_R^4$. The morphism $\pi_{R}'$ is smooth and proper by \cite[Lem. 2.1]{Katz-Another-Look-at-the-Dwork-Family}, and we refer to it as the \emph{Dwork family}. We define
        \[
        H=\{(\zeta_1,\dots,\zeta_5)\in\mu_{5,R}^5\mid\zeta_1\cdots\zeta_5=1\}/\Delta,
        \]
        where $\Delta\simeq \mu_{5, R}\hookrightarrow\mu_{5,R}^5$ is the diagonal subgroup. The group $H$ acts on the fibers of the Dwork family. The restriction of $\pi_R'$ over $\mathbb{P}^1_R\setminus(\{0,\infty\}\cup\mu_{5,R})$ is the pullback of a smooth proper morphism $\pi_R\colon Y\to\mathbb{P}^1_R\setminus\{0,1,\infty\}$ along the fifth power map. Moreover, the family $\pi_R$ also admits an action by $H$, see \cite[p. 102]{Katz-Another-Look-at-the-Dwork-Family}.
        
        We now specialize to the situation where $R=\mathbb{F}_q$ is a finite field of characteristic $p\neq 5$. For brevity, we write $\pi$ for the family $\pi_{\mathbb{F}_q}$ over $\mathbb{F}_q$. Let us define the $\Qlbar$-local system on $\mb{P}^1_{\mb{F}_q}\setminus \{0,1,\infty\}$ by 
        \begin{equation}\label{eqn: mirrorquintic}
            \mathcal{E}:=(\mathrm{R}^3\pi_{\textnormal{\'{e}t},\ast}\overline{\Q}_\ell)^H.
        \end{equation}

        \begin{rmk}
            It is expected that $\mc{E}$ is isomorphic to $R^3\check{\pi}_{*}\Qlbar$ for another family of Calabi--Yau threefolds $\check{\pi}: Z \rightarrow \mb{P}^1_{\mb{F}_q}\setminus \{0,1,\infty\}$. This is known to be true in characteristic zero, where $Z$ can be taken to be the \emph{mirror quintic} family. However, the construction involves a  crepant resolution of the singular family $Y/H$ (see \cite[Thm. 4.2.2]{Batyrev-Dual-Polyhedra-and-Mirror-Symmetry-for-Calabi-Yau-Hypersurfaces-in-Toric-Varieties}, where this is done for a much more general class of Calabi--Yau varieties), which, as far as we can tell, has not been worked out in positive characteristic.
        \end{rmk}

        \begin{lem}\label{lma: monodromy-of-Dwork-family}
            The $\Qlbar$-local system $\mathcal{E}$ has geometric monodromy $\mathbf{Sp}_{4,\Qlbar}$.
        \end{lem}

        \begin{proof}
            All we have to do is compute the Zariski closure of the image of the associated representation $\pi_1^\textnormal{ét}(\mb{P}^1_{\Fpbar} \setminus \{0,1,\infty\},\bar x)\to\mathrm{GL}_{4}(\overline{\Q}_\ell)$. The statement then follows from \cite[Thm. 8.6]{Katz-Another-Look-at-the-Dwork-Family} using a specialization argument. 
        \end{proof}

        We write $X=\mathbb{P}_{\mathbb{F}_q}^1\setminus\{0,1,\infty\}$ and $\pi\colon Y\to X$ as above. Let $x\in X(\mathbb{F}_{q^r})$ and let $\bar x$ be a geometric point lying over $x$. Let $F_{\bar x}$ denote geometric $q^r$-Frobenius acting on $Y_{\bar x}$.  We set $V=\mathrm{H}^3_\text{ét}(Y_{\bar{x}},\Qlbar)$, as well as $V_p=\mathrm{H}^3_\mathrm{crys}(Y_x/W(\mathbb{F}_{q ^r}))\otimes_{W(\mathbb{F}_{q^r})}\Qpbar$, and write $\varphi$ for its crystalline Frobenius endomorphism. Notice that $V^{H}$ and $V_p^H$ inherit Frobenius actions from $V$ and $V_p$ respectively, as $H$ acts by algebraic maps. We define
        \[
            P_{\textnormal{ét}}(T)=\det(1-TF_{\bar x}\mid V^H)\text{~and~} P_{\textnormal{cris}}(T)=\det(1-T\varphi^{er}\mid V_p^H),
        \]
        where $q=p^e$.

        \begin{lem}\label{lem: etale-slopes-and-crystalline-slopes}
            The polynomial $P_{\textnormal{ét}}(T)$ has $\Z$-coefficients and moreover $P_{\textnormal{ét}}(T)=P_{\textnormal{cris}}(T)$.
        \end{lem}

        \begin{proof}
            Since $H$ is a finite group, the composition of the projection map and the natural inclusion
            \[
                \mathrm{pr}\colon V\to V^H\to V
            \]
            is induced by an algebraic cycle (with $\Q$-coefficients). Moreover, the endomorphism $F_{\bar x}\colon V\to V$ is induced by an algebraic cycle. The eigenvalues of the composition $F_{\bar x}\circ\mathrm{pr}$, again induced by an algebraic cycle, is the multi-set consisting of $0$ with multiplicity $\dim(V)-\dim(V^H)$, and the eigenvalues of Frobenius on $V^H$. By \cite[Thm. 2 (2)]{Katz-Messing-Some-Consequences-of-the-Riemann-Hypothesis-for-Varieties-over-Finite-Fields}, we see that $P_\textnormal{ét}(T)$ has $\mb{Z}$-coefficients and is moreover independent of the choice of a Weil cohomology theory. In particular, we deduce that $P_{\textnormal{ét}}(T)=P_{\textnormal{cris}}(T)$.
        \end{proof}

        \begin{lem}\label{lemma: mirror-quintic-ord}
            For every valuation $v$ of $\overline{\Q}$ with $v(p)=1$, the generic $v$-slopes of $\mc{E}$ are $0,1,2,3$.
        \end{lem}

        The following proof was suggested to us by one of the referees; in the previous version we had a somewhat cumbersome argument using the Colmez--Fontaine theorem on weakly admissible filtered $\varphi$-modules. We thank the referee for this simplification.

        \begin{proof} 
            By \cref{lem: etale-slopes-and-crystalline-slopes}, we deduce that in order to compute the generic $v$-slopes of $V^H$, we are reduced to computing the generic slopes of the action of $\varphi^{er}$ on $V_p^H$.
            
            Thanks to \cite[Thm. 2.2]{Yu-Variation-of-the-Unit-Root-Along-the-Dwork-Family-of-Calabi-Yau-Varieties} the Dwork family is generically ordinary, so assume that $Y_x$ is ordinary. Applying \cite[Prop. 1.3.2]{Deligne-Cristaux-ordinaires-et-coordonnees-canoniques} there exists a decomposition
            \[
          V_p=\bigoplus_{i\geq 0}V_p^{\varphi=p^i}\otimes_{\Qp}\Qpbar
            \]
            such that for every $j\geq 0$, we have
            \begin{equation}\label{eqn:ordinary-decompose}    \bigg(\bigoplus_{i<j}V_p^{\varphi=p^i}\otimes_{\Qp}\Qpbar\bigg)\oplus\bigg(\mathrm{Fil}^j\otimes_{W(\mathbb{F}_{q^r})[1/p]}\Qpbar\bigg)=V_p,
            \end{equation}
            where $\mathrm{Fil}^j$ is the Hodge filtration induced by any choice of lift $Y_{\tilde{x}}$ of $Y_x$ to $W(\mathbb{F}_{q^r})$. Since the action of $H$ lifts to $Y_{\tilde{x}}$, it follows that $(V_p^H, \varphi|_{V_{p}^H}, \Fil^{\bullet}|_{V_{p}^H})$ is a direct summand of $(V_p, \varphi, \Fil^\bullet)$ as filtered $\varphi$-modules,  we deduce that  $V_p^H$ is also ordinary in the sense that there are analogous decompositions as \eqref{eqn:ordinary-decompose}. It follows that the slopes of Frobenius acting on $V_p^H$ are precisely the Hodge numbers of $(V_p^H, \Fil^{\bullet})$. It remains to show that the latter Hodge numbers are $0,1,2,3$: this is a well-known computation--see for example \cite[Lem.~1.5]{Harris-Shepherd-Barron-Taylor-A-Family-of-Calabi-Yau-Varieties-and-Potential-Automorphy}. 
        \end{proof}

        \begin{prop}\label{prop: F-isoc-of-Dwork-family-does-not-come-from-ab-var}
            The $\Qlbar$-local system $\mathcal{E}$ does not come from a family of abelian varieties.
        \end{prop}

        \begin{proof}
            By \cref{lma: monodromy-of-Dwork-family}, we have $\mathfrak{g}^{\rm{geo}}_{\mc{E}}\simeq\mathfrak{sp}_{4,\Qlbar}$. Note that there exists a canonical isomorphism of $\Qlbar$-local systems
            \[
                \rm{det}(\mc{E})^{\otimes-\frac{1}{2}}\otimes \rm{Sym}^2(\mc{E})\simeq\underline{\mathfrak{g}}_{\mathcal{E}}^\mathrm{geo}.
            \]
            Thus, it suffices to compute the generic slopes of
            \[
                \rm{det}(\mc{E})^{\otimes-\frac{1}{2}}\otimes \rm{Sym}^2(\mc{E})\simeq\rm{Sym}^2(\mc{E})(3).
            \]
            By \cref{lemma: mirror-quintic-ord}, the generic slopes of $\mc{E}$ are $0,1,2,3$, each occuring with multiplicity one. Therefore, the generic slopes of $\rm{Sym}^2(\mc{E})(3)$ are as follows:

            \begin{center}
                \begin{tabular}{ c|c|c|c|c|c|c|c }
                    slope & $-3$ & $-2$ & $-1$ & $0$ & $1$ & $2$ & $3$\\
                    \hline
                    multiplicity & $1$ & $1$ & $2$ & $2$ & $2$ & $1$ & $1$
                \end{tabular}
            \end{center}

            Applying \cref{prop: criterion-crystalline} shows that $\mc{E}$ cannot come from a family of abelian varieties.
        \end{proof}

        \begin{rmk}
            As pointed out by the referee, since our criterion is checked generically on  $\mb{P}^1_{\mb{F}_q}\setminus \{0,1,\infty\}$,  we actually get the technically stronger statement that for any Zariski dense open $U\subset \mb{P}^1_{\mb{F}_q}\setminus \{0,1, \infty\}$, $\mc{E}|_{U}$ does not come from a family of abelian varieties. 
        \end{rmk}

        \begin{rmk}[More examples of local systems violating \cref{prop: criterion-crystalline}]\label{rmk: more-examples}
            The following example was pointed out to us by the referee. Suppose $p\geq 5$, and  take $\mc{V}$ to be any of the local systems in  \cite[Prop. A.2.1]{Drinfeld-Kedlaya-Slopes-of-Indecomposable-F-Isocrystals}. Then $\mc{V}$ is a rank three local system on $\mb{P}^1_{\mb{F}_p}\setminus \{0,1,\infty\}$. It has  generic slopes $0,1,2$, by Cor. A.3.2(e) of loc.cit.. Moreover,  it has maximally unipotent monodromy at $0\in \mb{P}^1$ (see Remark A.1.4 of loc.cit.), and therefore the connected component of its geometric monodromy has Lie algebra  $\mathfrak{sl}_3$, which has generic slopes
            
            \begin{center}
                \begin{tabular}{ c|c|c|c|c|c }
                    slope & $-2$ & $-1$ & $0$ & $1$ & $2$ \\
                    \hline
                    multiplicity & $1$ & $2$ & $2$ & $2$ & $1$
                \end{tabular},
            \end{center}
            and  we therefore conclude by   \cref{prop: criterion-crystalline}.

            In fact, this example is closely related to the one we gave in \cref{prop: F-isoc-of-Dwork-family-does-not-come-from-ab-var}: they are both examples of \emph{hypergeometric local systems}.
        \end{rmk}
    
    \subsection{General case}
        We now come to the proof of \cref{thm: main-result}. 

        \begin{proof}[Proof of \cref{thm: main-result}]
            Take any finite separable morphism $C\to\mathbb{P}^1_{\mathbb{F}_q}$. Let $S\subseteq C$ be the finite set of points mapping to $\{0,1,\infty\}$ and consider the induced morphism $f\colon C\setminus S\to\mathbb{P}^1_{\mathbb{F}_q}\setminus\{0,1,\infty\}$. Write $\mc{F}=f^\ast\mathcal{E}$, where $\mc{E}$ is the $\Qlbar$-local system \eqref{eqn: mirrorquintic} (if $p\neq 5$) or the $\Qlbar$-local system from \cref{rmk: more-examples} (if $p\neq 2,3$). For every closed point $x$ of $C\setminus S$ we have $\underline{\mathfrak{g}}_{\mathcal{F},\bar x}\simeq \underline{\mathfrak{g}}_{\mathcal{E},\overline{f(x)}}$, where $\bar x,\overline{f(x)}$ denote  geometric points lying over $x,f(x)$ respectively. Hence $\mathcal{F}$ does not come from a family of abelian varieties over $C\setminus S$ by \cref{prop: criterion-crystalline} and \cref{prop: F-isoc-of-Dwork-family-does-not-come-from-ab-var}, respectively \cref{rmk: more-examples}.
        \end{proof}
	\printbibliography
\end{document}